\documentclass{amsart}

\usepackage{amsfonts,amsmath,amsthm,amssymb,amscd,latexsym,enumerate,etoolbox,mathrsfs,comment}

\usepackage[all,cmtip,pdf]{xy}
\usepackage{thmtools, thm-restate}
\usepackage{tikz-cd}
\usepackage{tikz}
\usetikzlibrary{shapes,arrows,decorations.markings}
\usetikzlibrary{calc}

\tikzset{->-/.style={decoration={
  markings,
  mark=at position .525 with {\arrow{Straight Barb}}},postaction={decorate}}}
\tikzset{->>-/.style={decoration={
  markings,
  mark=at position 0.5 with {\arrow{Straight Barb}},
  mark=at position 0.6 with {\arrow{Straight Barb}}},postaction={decorate}}}

\newcommand{\R}{\mathbb{R}}
\newcommand{\C}{\mathbb{C}} 
\newcommand{\N}{\mathbb{N}}

\newcommand{\Z}{{\mathbb Z}}
\newcommand{\Q}{{\mathbb Q}}

\newcommand{\Per}{{\bf P}}
\renewcommand{\phi}{\varphi}
\newcommand{\xn}{{\bf x}}
\newcommand{\yn}{{\bf y}}

\newcommand{\zn}{{\bf z}}

\theoremstyle{plain}
    \newtheorem{theorem}{Theorem}[section]
    \newtheorem{lemma}[theorem]{Lemma}
    
    \newtheorem{proposition}[theorem]{Proposition}
    
       \newtheorem{question}[theorem]{Question}
\theoremstyle{definition}
    \newtheorem{definition}[theorem]{Definition}
    \newtheorem{example}[theorem]{Example}

    \newtheorem{remark}[theorem]{Remark}
    
\theoremstyle{remark}

\begin{document}

\title[Smale spaces associated to flat manifolds]{Invariants for the Smale space associated to an expanding endomorphism of a flat manifold}

\author[Chaiser et al]{Rachel Chaiser}
\address{RC, MC, RJD, AK, JG, RH, LL, MR, AS,   Department of Mathematics,
University of Colorado Boulder
Campus Box 395,
Boulder, CO 80309-0395, USA }
\email{rachel.chaiser@colorado.edu, m.coateswelsh@gmail.com, robin.deeley@colorado.edu, annika.fahrner@colorado.edu, jamal.giornazi@colorado.edu, robi.huq@colorado.edu, levi.lorenzo@colorado.edu, maggie.reardon@colorado.edu, andrew.stocker@colorado.edu  }
\author[]{Maeve Coates-Welsh}
\author[]{Robin J. Deeley}
\author[]{Annika Farhner}
\author[]{Jamal Giornozi}
\author[]{Robi Huq}
\author[]{Levi Lorenzo}
\author[]{Jos\'e  Oyola-Cortes}
\address{JO,  Department of Mathematics
University of Puerto Rico, Rio Piedras Campus
17 University Ave. Ste 1701
San Juan PR, 00925-2537  }
\email{jose.oyola5@upr.edu}
\author[]{Maggie Reardon}
\author[]{Andrew M. Stocker}

\thanks{This work was partially supported by NSF Grant DMS 2000057 and by the CU Boulder Department of Mathematics in the context of its internal Research For Undergraduates program. JG was partially supported by a McNair scholarship and JO was partially supported by CU Boulder's Summer Multicultural Access to Research Training (SMART) program.}

\begin{abstract}
We study invariants associated to Smale spaces obtained from an expanding endomorphism on a (closed connected Riemannian) flat manifold. Specifically, the relevant invariants are the $K$-theory of the associated $C^*$-algebras and Putnam's homology theory for Smale spaces. The latter is isomorphic to the groupoid homology of the groupoids used to construct the $C^*$-algebras. 
\end{abstract}

\maketitle

\section*{Introduction}

A Smale space, $(X, \varphi)$, is a compact metric space, $X$, and a homeomorphism, $\varphi: X \rightarrow X$, that is uniformly hyperbolic; see Definition \ref{SmaSpaDef} for the precise definition. Associated to a Smale space are three \'etale groupoids: the stable, unstable, and homoclinic groupoids. Furthermore, using \cite{Ren}, there are C*-algebras associated with each of these groupoids. 

An important problem is the computation of invariants of a Smale space, often for ones constructed using the the groupoids/C*-algebras mentioned above. Three examples of such invariants are as follows:
\begin{enumerate}
\item the stable/unstable Putnam homology of $(X, \varphi)$ \cite{PutHom}, 
\item the homology of the relevant \'etale groupoids \cite{MR1752294}, and
\item the $K$-theory of the associated $C^*$-algebras.
\end{enumerate}

By \cite[Theorem 4.1 and Remark 4.4]{PV2} (also see \cite{PV1, PV3}) the Putnam homology and groupoid homology are isomorphic, so we need only compute two of three invariants in the previous list. In addition, it is worth noting that the stable and unstable groupoids depend on a choice of a finite set of periodic points; however, different choices lead to Morita equivalent groupoids and the above invariants are Morita invariant so they do not depend on the choice of periodic points. 

One natural way to construct a Smale space is as a solenoid. The process is as follows. Given a metric space, $Y$, and a continuous surjective map $g: Y \rightarrow Y$, one forms the space
\[
X:= \varprojlim (Y, g) = \{ (y_n)_{n\in \N} = (y_0, y_1, y_2, \ldots ) \: | \: g(y_{i+1})=y_i \hbox{ for each }i\ge0 \}
\]
and $\varphi: X \rightarrow X$ is defined via
\[
\varphi(x_0, x_1, x_2, \ldots ) = (g(x_0), g(x_1), g(x_2), \ldots) = (g(x_0), x_0, x_1, \ldots ).
\] 
The pair $(X, \varphi)$ is a dynamical system and one would like to know when it is a Smale space. In this context, Williams defined an important class of solenoids that are Smale spaces \cite{Wil} and Wieler provided the definitive result \cite{WiePhD, Wie}. Wieler gives conditions on $(Y, g)$ that ensure that $(X, \varphi)$ is a Smale space and likewise conditions on $(X, \varphi)$ that ensure that there exists $(Y, g)$ leading to $(X, \varphi)$ via the solenoid construction. The reader can see \cite[Theorems A and B]{Wie} for the precise result.

The goal of the present paper is to compute a number of invariants associated to a Smale space in the particular case when the Smale space is the solenoid associated to an expanding endomorphism of a flat manifold. These solenoids are the most well-behaved examples of Williams' solenoids. Here, by compute we mean that the relevant Smale space invariant is completely determined by the algebraic topology of $Y$ and $g$. Using these results we consider a number of explicit computations.

The main theoretical results of the present paper are the following two theorems (see the main body of the paper for further details on the notation used):

\begin{theorem}
Suppose that $G^s(P)$ is the stable groupoid obtained from an expanding endomorphism $g : Y \rightarrow Y$ where $Y$ is a flat manifold. Then 
\[  K_*(C^*(G^s(P))) \cong \lim ( K^*(Y), t_{{\rm K-theory}}) \hbox { and }  H_*(G^s(P)) \cong \lim ( H^*(Y), t_{{\rm cohomology}}) \]
where $t_{{\rm K-theory}}$ is the transfer map in $K$-theory associated to $g$ and $t_{{\rm cohomology}}$ is the transfer map in cohomology associated to $g$.
\end{theorem}

\begin{theorem} 
Suppose that $G^u(P)$ is the unstable groupoid obtained from an expanding endomorphism $g : Y \rightarrow Y$ where $Y$ is a flat manifold. Then 
\[  K_*(C^*(G^u(P))) \cong \lim ( K_*(Y), t_{{\rm K-homology}}) \hbox { and }  H_*(G^u(P)) \cong \lim ( H_*(Y), t_{{\rm homology}}) \]
where $t_{{\rm K-homology}}$ is the transfer map in $K$-homology associated to $g$ and $t_{{\rm homology}}$ is the transfer map in homology associated to $g$.
\end{theorem}

The proof of the first of these theorems uses known results about the decomposition of the stable groupoid and hence the proof given here is quite short. The one thing missing from previous work is the precise description of the connecting maps in the inductive limits. The proof of the second theorem given here is more detailed. The result is obtained by constructing a Morita equivalence between $G^u(P)$ and the action groupoid of the odometer action associated to $(Y, g)$. In the recent preprint \cite{PV3} it is shown that this Morita equivalence follows from work of Nekrashevych \cite{MR2274718} (see \cite[Example 3.4]{PV3} for details). Nevertheless, we have included our proof as it is quite explicit, self-contained, and uses only basic covering space theory.  

Based on the structure of the previous theorem, it is useful, at least in part, to understand the transfer maps in $K$-theory, $K$-homology, and (co)homology. Our main result in this regard is the following theorem (see the main body of the paper for further details on the notation used):

\begin{theorem}
Suppose $Y$ is a flat manifold. Then there exists an expanding endomorphism $g: Y \rightarrow Y$ such that the transfer map on (co)homology is an isomorphism on the torsion subgroups. In particular, for this $(Y, g)$, we have that
\[
T(H_*(Y)) \cong T(H_*(G^u(P))) \hbox{ and } T(H^*(Y)) \cong T(H_* (G^s(P)))
\]
where
\begin{enumerate}
\item $(X, \varphi)$ is the Smale space associated to $(Y,g)$ and
\item $T(G)$ denotes the torsion subgroup of an abelian group $G$.
\end{enumerate}
\end{theorem}

In other words, any torsion groups that appear in the (co)homology of a flat manifold also appear in the torsion of Putnam's homology. Since there are quite a few results on the (co)homology of flat manifolds this leads to new information about the possible torsion subgroups that can appear in Putnam's homology.

A number of explicit computations are considered here. Notably, we give an explicit counterexample to Question 8.3.2 in \cite{PutHom} and give some positive results related to this question. In recent independent work, Proietti and Yamashita have reformulated Question 8.3.2 to take orientation into account and have given a positive answer to this reformulation, see \cite{PV3} for details. Their result applies to the homology of the unstable groupoid of a solenoid, but does not apply to the homology of the stable groupoid of solenoids (which are discussed in the present paper). Their general result is consistent with our results and we would be remiss to not mention that their beautiful result is more general than ours in the context of the homology of the unstable groupoid. It applies in particular to any Williams' solenoid, not just ones associated with flat manifolds. 

The structure of the paper is as follows. The preliminaries are discussed in Section \ref{secPrelim}. This includes a discussion of Smale spaces in general, those associated to flat manifolds via the solenoid construction, and some basic results about the transfer map. Section \ref{KtheoryHomologySec} contains our main structural results, which realized the various invariants discussed above as inductive limits. Explicit computations are discussed in Section \ref{SecComputations}. The main highlights are the torsion result mentioned in the previous paragraph, examples in low dimension, and a detailed application of the general results to Hantzsche--Wendt manifolds (a special class of flat manifolds). Finally, it is worth mentioning that (while we have restricted to flat manifolds) our results apply to the general case of expanding endomorphisms consider in \cite{ShubExp}. The reason for restricting to flat manifolds is based on the main result of \cite{EpSh}, which states that each flat manifold admits at least one expanding endomorphism.

A short summary of the notation used in the present paper is as follows:
\begin{enumerate}
\item If $H$ is a group and $\alpha: H \rightarrow H$, then $\lim (H, \alpha)$ denotes the inductive limit group associated with the stationary inductive limit:
\[ H \rightarrow H \rightarrow H \rightarrow \ldots \]
\item $Y$ denotes a closed, connected, Riemannian flat manifold of dimension $d$. We will refer to $Y$ as simply a flat manifold. 
\item $g: Y \rightarrow Y$ is a expanding endomorphism in the sense of Shub, see page 176 of \cite{ShubExp} or Section \ref{ExpEndSec}.
\item $(X, \varphi)$ is a Smale space. Typically, it is the Smale space associated $(Y, g)$, which is constructed via an inverse limit, see Section \ref{ExpEndSec}.
\item $y_0$ is a particular fixed point of $g$ and $P=(y_0,y_0, \ldots)$ is a fixed point of $\varphi$.
\item $G^s(P)$ and $G^u(P)$ respectively denote the stable and unstable groupoid associated to $(X, \varphi)$ with respect to the set $\{ P \}$.
\item If $G$ is an amenable groupoid, then $C^*(G)$ denotes its $C^*$-algebra; all groupoids in the present paper are amenable.
\item $H^s_*(X, \varphi)$ and $H^u_*(X, \varphi)$ denote Putnam's homology for $(X, \varphi)$.
\item If $Z$ is a compact Hausdorff space, then $H_*(Z)$ and $H^*(Z)$ respectively denote its Cech homology and cohomology.
\item If $G$ is an \'etale groupoid, then $H_*(G)$ denotes its groupoid homology, see \cite{MR1752294}. 
\item If $Y$ is a space viewed as a trivial groupoid, then the groupoid homology of $Y$ is natural isomorphic to the cohomology of $Y$, see \cite[Section 3.5]{MR1752294} for details.
\end{enumerate}
\section{Preliminaries} \label{secPrelim}

\subsection{Smale spaces} \label{SmaleSec}
\begin{definition} \label{SmaSpaDef}
A Smale space is a metric space $(X, d)$ along with a homeomorphism $\varphi: X\rightarrow X$ with the following additional structure: there exists global constants $\epsilon_X>0$ and $0< \lambda < 1$ and a continuous map, called the bracket map, 
\[
[ \ \cdot \  , \ \cdot \ ] :\{(x,y) \in X \times X : d(x,y) \leq \epsilon_X\}\to X
\]
such that the following axioms hold
\begin{itemize}
\item[B1] $\left[ x, x \right] = x$;
\item[B2] $\left[x,[y, z] \right] = [x, z]$ when both sides are defined;
\item[B3] $\left[[x, y], z \right] = [x,z]$ when both sides are defined;
\item[B4] $\varphi[x, y] = [ \varphi(x), \varphi(y)]$ when both sides are defined;
\item[C1] For $x,y \in X$ such that $[x,y]=y$, $d(\varphi(x),\varphi(y)) \leq \lambda d(x,y)$;
\item[C2] For $x,y \in X$ such that $[x,y]=x$, $d(\varphi^{-1}(x),\varphi^{-1}(y)) \leq \lambda d(x,y)$.
\end{itemize}
We denote a Smale space simply by $(X,\varphi)$.
\end{definition}

An introduction to Smale spaces can be found in \cite{Put}. Throughout we assume that $X$ is an infinite set. The Smale spaces we consider here will be of a special form. They will be solenoids, and we will discuss this in detail in the next section. Before doing so, a few general facts will be discussed.

\begin{definition}
Suppose $(X, \varphi)$ is a Smale space and $x$, $y$ are in $X$. Then we write $x \sim_s y$ (respectively, $x\sim_u y$) if $\lim_{n \rightarrow \infty} d(\varphi^n(x), \varphi^n(y)) =0$ (respectively, $\lim_{n\rightarrow \infty}d(\varphi^{-n}(x) , \varphi^{-n}(y))=0$). The $s$ and $u$ stand for stable and unstable respectively.
\end{definition}

Given $x\in X$, the global stable and unstable set of $x$ are defined as follows:
\[
X^s(x) = \{ y \in X \: | \: y \sim_s x \} \hbox{ and } X^u(x)=\{ y \in X \: | \: y \sim_u x \}.
\]
Given, $0< \epsilon \le \epsilon_X$, the local stable and unstable set of a point $x \in X$ are defined as follows:
\begin{align}
X^s(x, \epsilon) & = \{ y \in X \: | \: [x, y ]= y \hbox{ and }d(x,y)< \epsilon \} \hbox{ and } \\
X^u(x, \epsilon) & = \{ y \in X \: | \: [y, x]= y \hbox{ and } d(x,y)< \epsilon \}.
\end{align} 
An important fact relating the local and global sets is the following:
\[ \displaystyle X^s(x) = \bigcup_{n\in \N} \varphi^{-n}(X^s(\varphi^n(x), \epsilon)) \hbox{ and }\displaystyle X^u(x) = \bigcup_{n\in \N} \varphi^n(X^u(\varphi^{-n}(x), \epsilon)). \]
It is worth noting that these unions are nested, that is,
\[ X^s(x, \epsilon) \subseteq \varphi^{-1}(X^s(\varphi(x), \epsilon) \subseteq \varphi^{-2}(X^s(\varphi^2(x), \epsilon) \subseteq \ldots
\]
and likewise in the decomposition of $X^u(x)$. The topologies on the local stable and local unstable sets are defined to be the subspace topology. While the topologies on $X^s(x)$ and $X^u(x)$ are defined using these decompositions, see  \cite[Theorem 2.10]{Kil} for details. With these topologies, $X^s(x)$ and $X^u(x)$ are locally compact and Hausdorff.

\subsection{Groupoids and C$^*$-algebras associated to Smale spaces} \label{subSecGroupoids}
The construction of the groupoids associated to a Smale space is considered in this section. We will follow \cite{PutSpi} for it. To begin, fix a finite $\varphi$-invariant set of periodic points of $\varphi$, which we denote by $\Per$. If $(X, \varphi)$ has a fixed point (which will be the case for the solenoids we consider) one can take $\Per$ to be the set containing just the fixed point. Define
\[
X^u(\Per):=\{ x \in X \: | \: x \sim_u p \hbox{ for some }p \in \Per \}
\]
and
\[
G^s(\Per) := \{ (x, y) \in X^u(\Per) \times X^u(\Per) \: | \: x \sim_s y \}.
\] 
Notice if $\Per=\{ P \}$ where $P$ is a fixed point of $(X, \varphi)$, then 
\[
X^u(\Per)=X^u(P).
\]
There is a topology, see \cite{PutSpi}, turning $G^s(\Per)$ into an \'etale groupoid, which is amenable. The groupoid $C^*$-algebra associated to $G^s(\Per)$ is denoted by $C^*(G^s(\Per))$ (see \cite{Ren} for the construction). In a similar way there is the construction of the unstable groupoid and its $C^*$-algebra, which are denoted by $G^u(\Per)$ and $C^*(G^u(\Per))$.

We will only need to consider the topology on $G^s(\Per)$ (and likewise for $G^u(\Per)$) in detail in Section \ref{KtheoryHomologySec}. A basic neighborhood for this topology is given as follows. As inputs, there are $x\in X^u(\Per)$, $y\in X^u(\Per)$, $N\in \N$, and $\delta>0$ where
\begin{enumerate}
\item $y \sim_s x$,
\item $\varphi^N(x) \in X^s(\varphi^N(y), \epsilon_X)$, and
\item $\varphi^N(X^u(x, \delta)) \subseteq X^u(\varphi^N(x), \epsilon_X)$.
\end{enumerate} 
The basic neighborhood associated to these inputs is the set
\[
V = \{ (h(z), z) \mid z \in X^u(x, \delta) \}
\]
where $h: X^u(x, \delta) \rightarrow X^u(y, \epsilon_X)$ is defined via
\[
z \mapsto \varphi^{-N} [ \varphi^N(z), \varphi^N(y)].
\]
For more on this topology and the analogous one on $G^u(\Per)$, see any of \cite{PutSpi}, \cite{Kil}, \cite{MR4138909}.

The $K$-theory of these $C^*$-algebras is denoted respectively by $K_*(C^*(G^s(\Per)))$ and $K_*(C^*(G^u(\Per)))$. These $K$-theory groups are important invariants of the Smale space and computing them for a particular class of examples is one of the main goals of this paper. Another important invariant is the homology of the groupoids $G^s(\Per)$ and $G^u(\Per)$; the homology of an \'etale groupoid is defined in \cite{MR1752294}. The other main goal of the present paper is computing this homology.
 
\subsection{Expanding endomorphisms on flat manifolds} \label{ExpEndSec}
A flat manifold refers to a closed, connected, Riemannian flat manifold. Throughout, $Y$ is a flat manifold of dimension $d$. Examples of flat manifolds include the circle, the torus and the Klein bottle; see \cite{MR862114} for more details and many more examples (see in particular page 41 of \cite{MR862114}). A fundamental property of flat manifolds is the following: the fundamental group of $Y$, $\pi_1(Y)$, is torsion-free and fits within the following short exact sequence
\[ 0 \rightarrow \Z^d \rightarrow \pi_1(Y) \rightarrow F \rightarrow 0 \]
where $\Z^d$ is maximal abelian and $F$ is a finite group, which is called the holonomy.

The following result is well-known, see for example \cite[Lemma 2.7]{MR2581917} for more details.
\begin{proposition} \label{flatManTorsionHomology}
If $x \in T(H_*(Y))$ or $T(H^*(Y))$, then the order of $x$ divides $|F|$. In particular, for any $k \in \N$ and $x \in T(H_*(Y))$ or $T(H^*(Y))$, $(|F|+1)^k x =x$.
\end{proposition}

Let $g: Y \rightarrow Y$ be an expanding endomorphism. That is (see page 176 of \cite{ShubExp}) there exists $C>0$ and $\lambda>1$ such that $|| Tg^k v || \ge C\lambda^k ||v ||$ for each $v \in TY$ and strictly positive integer $k$. Here $|| \cdot ||$ denotes a fixed Riemannian metric, but it is worth noting that being expanding is independent of the choice of metric (although the particular constants $C$ and $\lambda$ do depend on the metric). Furthermore, by \cite{EpSh}, for any flat manifold $Y$ there exists at least one expanding endomorphism on $Y$. By \cite[Proposition 3]{ShubExp}, $g$ is a covering map and, since $Y$ is compact, $g$ is an $n$-fold cover for some $n \ge 2$.

A few examples might be useful for the reader.

\begin{example}
Let $S^1$ denote the unit circle in the complex number and $n\ge 2$ be an integer. Then $g: S^1 \rightarrow S^1$ defined via $z \mapsto z^n$ is an expanding endomorphism.
\end{example}
\begin{example}
The $9$-fold cover determined by the following diagrams is an expanding endomorphism of the Klein bottle:
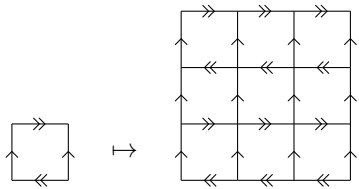
\begin{figure}[h]
\begin{tikzpicture}[scale=0.75]
\draw[->-] (0,0)--(0,1);
\draw[->-] (1,0)--(1,1);
\draw[->>-] (1,0)--(0,0);
\draw[->>-] (0,1)--(1,1);

\node at (2,1/2) {$\mapsto$};

\draw[->-] (3,0)--(3,1); \draw[->-] (3,1)--(3,2); \draw[->-] (3,2)--(3,3);
\draw[->-] (6,0)--(6,1); \draw[->-] (6,1)--(6,2); \draw[->-] (6,2)--(6,3);
\draw[->>-] (3,3)--(4,3); \draw[->>-] (4,3)--(5,3); \draw[->>-] (5,3)--(6,3);
\draw[->>-] (6,0)--(5,0); \draw[->>-] (5,0)--(4,0); \draw[->>-] (4,0)--(3,0);

\draw[->-] (4,0)--(4,1); \draw[->-] (4,1)--(4,2); \draw[->-] (4,2)--(4,3);
\draw[->-] (5,0)--(5,1); \draw[->-] (5,1)--(5,2); \draw[->-] (5,2)--(5,3);

\draw[->>-] (3,1)--(4,1); \draw[->>-] (4,1)--(5,1); \draw[->>-] (5,1)--(6,1);
\draw[->>-] (6,2)--(5,2); \draw[->>-] (5,2)--(4,2); \draw[->>-] (4,2)--(3,2);
\end{tikzpicture}
\caption{A nine fold self-cover of the Klein bottle}
\label{dia9Klein}
\end{figure}



\end{example}

Returning to the general case, associated to $(Y, g)$ there is a Smale space, see \cite{ShubExp, Put, Wil, Wie}. The construction is as follows. Let 
\[
X:= \varprojlim (Y, g) = \{ (y_n)_{n\in \N} = (y_0, y_1, y_2, \ldots ) \: | \: g(y_{i+1})=y_i \hbox{ for each }i\ge0 \}
\]
and $\varphi: X \rightarrow X$ be defined via
\[
\varphi(x_0, x_1, x_2, \ldots ) = (g(x_0), g(x_1), g(x_2), \ldots) = (g(x_0), x_0, x_1, \ldots ).
\] 
There is some flexibility in the specific metric on $X$, see \cite{WiePhD, Wie} for a particular choice. The topology is the subspace topology obtained by viewing $X$ as a subset of a countable product of copies of $Y$ with the product topology.

By \cite[Theorem 1 and Lemma 3]{ShubExp}, $g$ has a (unique) fixed point that lifts to a fixed point in the universal cover of $Y$. It will be denoted by $y_0$. We will use this as our based point, so $\pi_1(Y)$ denotes $\pi_1(Y, y_0)$. The fixed point of $\varphi$, given by $(y_0, y_0, y_0, \ldots)$, will be used as the finite $\varphi$-invariant set of periodic points of $\varphi$. That is, we take $\Per$ to be $\{ P \}$ where $P=(y_0, y_0, \ldots)$ in the construction of the groupoid $C^*$-algebras associated to $(X, \varphi)$. A summary of the situation is as follows:

\begin{definition} \label{solenoidGroupoidDefn}
The Smale space associated to $(Y, g)$ is defined to be the Smale space $(X, \varphi)$ discussed in the previous paragraphs. The stable (respectively unstable) groupoid associated to $(X, \varphi)$ is denoted by $G^s(P)$ (respectively $G^u(P)$) where $P=(y_0, y_0, \ldots)$ and $y_0$ is the fixed point discussed above.
\end{definition}

\begin{proposition} \label{UniversalCoverProp}
Suppose $(Y, g)$, $(X, \varphi)$, and $P$ are as in the previous paragraphs. Then the map $X^u(P) \rightarrow Y$ defined via
\[
(z_0, z_1, z_2, \ldots ) \mapsto z_0
\]
is the universal cover of $Y$.
\end{proposition}
\begin{proof}
By \cite[Theorem 3.12]{MR3868019}, the map in the statement of the theorem is a covering map. By \cite{Wil}, $X^u(P, \epsilon) \cong \mathbb{D}_d$ where $\mathbb{D}_d$ is the open disk of dimension $d={\rm dim}(Y)$. Since 
\[ \displaystyle X^u(P) = \bigcup_{n\in \N} \varphi^n(X^u(P, \epsilon)) \]
we have that $X^u(P)$ is homeomorphic to $\R^d$ and hence the map is the universal cover.
\end{proof}

\subsection{The transfer map}

Suppose $X$ and $Z$ are closed manifolds and $f: X \rightarrow Z$ is a $n$-fold cover; we will be interested in the case when $X=Z$. Associated to $f$ there are induced maps on (co)homology:
\begin{align*}
f_* : H_*(X) \rightarrow H_*(Z) \\
f^*: H^*(Z) \rightarrow H^*(X) \\
\end{align*}
There are also induced maps on $K$-homology and $K$-theory. 

In addition to these induced maps, there are transfer maps associated to $f$. For (co)homology, the transfer maps will be denoted by
\begin{align*}
t_{{\rm homology}} : H_*(Z) \rightarrow H_*(X) \\
t_{{\rm cohomology}}: H^*(X) \rightarrow H^*(Z) \\
\end{align*}
Likewise there are transfer maps on $K$-homology and $K$-theory. We will use the following properties of the transfer map (see for example \cite{MR3073917} for details):
\begin{proposition} \label{propTransferProp}
The transfer map has the following properties
\begin{enumerate}
\item For both cohomology and $K$-theory, it is a rational surjection. For both homology and $K$-homology, it is a rational injection. The transfer map is a rational isomorphism when $f$ is a self-cover of a closed manifold.
\item For cohomology, $t_{{\rm cohomology}} \circ f^*$ is multiplication by $n$ and for homology, $f_* \circ t_{{\rm homology}}$ is multiplication by $n$.
\end{enumerate}
\end{proposition}

\subsection{Groupoids associated to covering maps}
Suppose $\pi : \tilde{Y} \rightarrow Y$ is a coving map. Then there is an \'etale groupoid associated to $\pi$:
\[
G(\pi)=\{ (w, z) \in \tilde{Y} \times \tilde{Y} \mid \pi(w)=\pi(z) \}.
\]
The relevant topology is the subspace topology from $G(\pi) \subseteq \tilde{Y}\times \tilde{Y}$. For more on this construction see \cite{MR3084709}. It is worth noting that since $\pi$ is a covering map, it is a local homeomorphism so it fits within the context of \cite{MR3084709}. 

We will assume the reader is familiar with the notion of Morita equivalence for $C^*$-algebras and related equivalences for groupoids. The reader can find details on this in \cite{MR873460} also see \cite{MR4030921}. The groupoid $G(\pi)$ is equivalent to $Y$ as a trivial groupoid using $\tilde{Y}$ as the middle space. The action of $G(\pi)$ on $\tilde{Y}$ is given by
\[
(w, z) \cdot \tilde{y} = \left\{ \begin{array}{ll} w & \hbox{when }z=\tilde{y} \\ \hbox{not defined } & \hbox{otherwise } \end{array} \right.
\]
and the action of $Y$ on $\tilde{Y}$ is given by
\[
\tilde{y} \cdot y = \left\{ \begin{array}{ll} \tilde{y} & \hbox{when }y=\pi(\tilde{y}) \\ \hbox{not defined } & \hbox{otherwise } \end{array} \right.
\]

The following situation will be of interest in the paper. Suppose that $\pi : \tilde{Y} \rightarrow Y$ is the universal cover of $Y$ and $p: Y_0 \rightarrow Y$ is a finite covering map. By properties of the universal cover, there is a covering map $\pi_0: \tilde{Y} \rightarrow Y_0$ that fits in the following commutative diagram:
\[
  \begin{tikzcd}
    \tilde{Y} \arrow{dr}{\pi} \arrow[swap]{d}{\pi_0}  \\
     Y_0 \arrow{r}{p} & Y
  \end{tikzcd}
\]
It follows that there is an open inclusion of groupoids, $G(\pi_0) \subseteq G(\pi)$. For more on this construction see \cite[Example 3.9]{MR4030921}

The next theorem is likely known, but we could not find a proof in the literature (although \cite{MR3073917} and \cite{MR918241} contain related statements).

\begin{theorem} \label{transferThmForStable}
Using the notation in the previous few paragraphs, there are commutative diagrams 
\[
  \begin{tikzcd}
    K_*(C^*(G(\pi_0))) \arrow{d}{{\rm ME}}  \arrow{r}{\iota_*} & K_*(C^*(G(\pi))) \arrow{d}{{\rm ME}}  \\
    K^*(Y) \arrow{r}{t_{K{\rm -theory}}}  & K^*(Y)
  \end{tikzcd}
\]
and
\[
  \begin{tikzcd}
    H_*(G(\pi_0)) \arrow{d}{{\rm ME}}  \arrow{r}{\iota_*} & H_*(G(\pi)) \arrow{d}{{\rm ME}}  \\
    H^*(Y) \arrow{r}{t_{cohomology}}  & H^*(Y)
  \end{tikzcd}
\]
where
\begin{enumerate}
\item $\iota_*$ is the map on $K$-theory/cohomology induced from the open inclusion $G(\pi_0) \subseteq G(\pi)$;
\item ${\rm ME}$ is the isomorphism on $K$-theory/cohomology induced from the Morita equivalences discussed in the paragraph before the statement of the theorem;
\item $t_{K{\rm -theory}}$ and $t_{cohomology}$ are the transfer maps on $K$-theory/cohomology associated to the cover $p: Y_0 \rightarrow Y$.
\end{enumerate}
\end{theorem}
\begin{proof}
We will only discuss the case of $K$-theory in detail. Let $\Gamma$ be the fundamental group of $Y$ and $\Gamma_0$ be the finite index subgroup associated to the finite order cover $p: Y_0 \rightarrow Y$. 

The transfer map on $K$-theory is discussed in \cite{MR3073917}. In particular, it is given (at the level of modules) by
\[ E \mapsto E \otimes_{C(Y_0)} \mathcal{T}^{Y_0}_Y \]
where $\mathcal{T}^{Y_0}_Y$ is defined as follows. As a vector space it is $C(Y_0)$. The action of $C(Y_0)$ is the standard one and the action of $C(Y)$ is obtained from the identification of $Y_0$ with $\tilde{Y} \times_{\Gamma} \Gamma/\Gamma_0$, which is the quotient of $\tilde{Y} \times \Gamma/\Gamma_0$ by the diagonal action. The interested reader can find more details just before Lemma 3.12 in \cite{MR3073917}. 

The maps on $K$-theory induced from the Morita equivalences are given by the tensor product with the completion of $C_c(\tilde{Y})$. We denote these by $\mathcal{E}_{\pi}$ and $\mathcal{E}_{\pi_0}$. We note that although the starting point for each is $C_c(\tilde{Y})$ the norms that we complete with respect to are different so they are different modules: $\mathcal{E}_{\pi}$ is a $C^*(G(\pi))$-$C(Y)$ module and $\mathcal{E}_{\pi_0}$ is a $C^*(G(\pi_0))$-$C(Y_0)$ module.

The map $\iota_*$ is induced from a $*$-homomorphism. Namely, the inclusion $C^*(G(\pi_0)) \subseteq C^*(G(\pi))$. Hence, $\iota_*$ is given by the tensor product with $C^*(G(\pi))$ viewed as a $C^*(G(\pi_0))$-$C^*(G(\pi))$ module. 

With these four modules introduced, the proof reduces to showing that 
\[
\mathcal{E}_{\pi_0} \otimes_{C(Y_0)} \mathcal{T}^{Y_0}_Y \cong C^*(G(\pi)) \otimes_{C^*(G(\pi))} \mathcal{E}_{\pi}.
\]
This follows from \cite[Example 15.5.2(a)]{Black} and by identifying 
\[
\mathcal{E}_{\pi_0} \otimes_{C(Y_0)} \mathcal{T}^{Y_0}_Y \cong \mathcal{E}_{\pi_0} \otimes_{C(Y_0)} C(Y_0) \cong \mathcal{E}_{\pi}
\]
as right $C(Y)$-modules.
\end{proof}

\begin{remark} \label{RemSpecialCoverTransferGroupoid}
A special case of the previous theorem is the following situation:
\[
  \begin{tikzcd}
    X^u(P) \arrow{dr}{\pi} \arrow[swap]{d}{\pi_0}  \\
     Y \arrow{r}{g} & Y
  \end{tikzcd}
\]
where $(Y, g)$, $X^u(P)$, etc are defined as in Subsection \ref{ExpEndSec} (see in particular Definition \ref{solenoidGroupoidDefn} and Proposition \ref{UniversalCoverProp}).
\end{remark}

\section{$K$-theory and homology} \label{KtheoryHomologySec}

\subsection{The stable groupoid/algebra} 

\begin{theorem} \label{KHstable}
Suppose that $G^s(P)$ is the stable groupoid obtained from an expanding endomorphism $g : Y \rightarrow Y$ where $Y$ is a flat manifold as in \ref{solenoidGroupoidDefn}. Then 
\[  K_*(C^*(G^s(P))) \cong \lim ( K^*(Y), t_{{\rm K-theory}}) \hbox { and }  H_*(G^s(P)) \cong \lim ( H^*(Y), t_{{\rm cohomology}}) \]
where $t_{{\rm K-theory}}$ is the transfer map in $K$-theory associated to $g$ and $t_{{\rm cohomology}}$ is the transfer map in cohomology associated to $g$.
\end{theorem}

To prove this theorem we used a decomposition of the stable groupoid, see \cite{ThoAMS} (also see \cite{MR3868019} along with \cite{MR4138909} for a generalization).

\begin{definition}
For each integer $k\ge 0$, let 
\[
G_k(P)= \{ (\xn, \yn) \in X^u(P) \times X^u(P) \mid g^k(x_0)=g^k(y_0) \}.
\]
\end{definition}
Using \cite[Theorem 3.12]{MR3868019}, it follows that the groupoids $G_k(P)$ have the following properties:
\begin{enumerate}
\item For each $k$, $G_k(P)$ is an \'etale groupoid with the subspace topology.
\item For each $k$, $G_k(P)$ is Morita equivalent to $Y$.
\item There are nested open inclusions: 
\[ G_0(P) \subseteq G_1(P) \subseteq G_2(P) \subseteq \ldots \]
\item As topological groupoids, $G^s(P) = \cup_{k \ge 0} G_k(P)$.
\end{enumerate}

From these properties, it follows that both the $K$-theory and homology of the stable algebra are inductive limits. Furthermore, it follows from Theorem \ref{transferThmForStable} that the connecting map is given by the relevant transfer map (see Remark \ref{RemSpecialCoverTransferGroupoid}).

\subsection{The unstable groupoid/algebra}

\begin{theorem} \label{KHunstable}
Suppose that $G^u(P)$ is the unstable groupoid obtained from an expanding endomorphism $g : Y \rightarrow Y$ where $Y$ is a flat manifold as in \ref{solenoidGroupoidDefn}. Then 
\[  K_*(C^*(G^u(P))) \cong \lim ( K_*(Y), t_{{\rm K-homology}}) \hbox { and }  H_*(G^u(P)) \cong \lim ( H_*(Y), t_{{\rm homology}}) \]
where $t_{{\rm K-homology}}$ is the transfer map in $K$-homology associated to $g$ and $t_{{\rm homology}}$ is the transfer map in homology associated to $g$.
\end{theorem}

To prove this theorem, we prove that the unstable groupoid is Morita equivalent to the orbit relation of the odometer action associated to $(Y, g)$. To do so, the odometer associated to $(Y, g)$ must be introduced. As above, using \cite[Theorem 1]{ShubExp}, $g$ has a fixed point $y_0$, which will be our based point (so that in particular, $\pi_1(Y)$ denotes $\pi_1(Y, y_0)$). Associated to $g$ is a chain of finite index, proper subgroup inclusions:
\[
\pi_1(Y) \supset g_*(\pi_1(Y)) \supset g^2_*(\pi(Y)) \supset \cdots
\] 
The associated odometer is obtained as follows. For the space, let
\[ \Omega = \varprojlim (\Omega_i, f_{i-1}^i) \]
where $\Omega_i= \pi_1(Y)/ g^i_*(\pi(Y))$ and $f_{i-1}^i$ is given by inclusion of cosets. Each $\Omega_i$ is a finite set (that contains more than one element) and hence $\Omega$ is a Cantor set. An element in $\Omega$ will be written as
\[ (\gamma_0 \pi_1(Y), \gamma_1 g_*(\pi_1(Y)), \gamma_2 g^2_*(\pi_1(Y)), \ldots ). \]
Next, the action of $\pi_1(Y)$ on $\Omega$ is given as follows:
\[
\gamma \cdot (\gamma_0 \pi_1(Y), \gamma_1 g_*(\pi_1(Y)),  \ldots ) =(\gamma \gamma_0 \pi_1(Y), \gamma \gamma_1 g_*(\pi_1(Y)),  \ldots ).
\]
where $\gamma \in \pi_1(Y)$. The space $\Omega$ with this action of $\pi_1(Y)$ is called the odometer action associated to $(Y, g)$. Let $R_{{\rm orbit}}$ be the orbit relation associated to this action. 

\begin{lemma}
Using the notation in the previous paragraphs, the groupoids $G^u(P)$ and $R_{{\rm orbit}}$ are Morita equivalent.
\end{lemma}
\begin{proof}
We must relate the odometer to the solenoid. To begin, fix a Riemannian metric on $Y$ and recall that $y_0$ is a fixed point of $g$. The Cantor set in the odometer action can be described in terms of preimages of $y_0$ with respect to $g$, $g^2$, etc. The starting point is the observation that there is a one-to-one correspondence between $g^{-1}(y_0)$ and cosets associated to the subgroup $g_*(\pi_1(Y))$. 

This correspondence is given as follows. Given a coset, take a loop, $\gamma$, based at $y_0$ representing a class in that coset. Let $\tilde{\gamma}: [0,1] \rightarrow Y$ denote the unique lift of $\gamma$ with respect to $g$ to a path starting at $y_0$. Then $\tilde{\gamma}(1) \in g^{-1}(y_0)$. Furthermore, this defines the required one-to-one correspondence. That this process is well-defined and one-to-one follows from elementary facts from covering space theory.

Repeating this process with $g^{-2}(y_0)$, $g^{-3}(y_0)$, etc, one has that $\Omega$ is homeomorphic to 
\[
Z:=\{ (y_0, y_1, y_2, \ldots ) \mid y_0 \hbox{ is the fixed point above and }g(y_{i+1})=y_i \} 
\]
where the topology is the subspace topology when considering $Z$ as a subset of $X^s(P)$. The explicit map constructed above will be denoted by $\Phi: \Omega \rightarrow Z$. One can check that it is continuous and its inverse (which is also continuous) is given by the following:
\[
\yn \in Z \mapsto (\gamma_0 \pi_1(Y), \gamma_1 g_*(\pi_1(Y)), \gamma_2 g^2_*(\pi_1(Y)), \ldots ) \in \Omega 
\] 
where $\gamma_i=g^i \circ f_i$ where $f_i$ is a path $f_i: [0,1] \rightarrow Y$ with $f(0)=y_0$ and $f(1)=y_i$.

Consider the restriction of the unstable relation to $Z$. That is, let 
\[
R_u=\{ (\xn, \yn) \in Z \times Z \mid \xn \sim_{u} \yn \}.
\] 
Our goal is to show that $\Phi \times \Phi$ is an isomorphism of topological groupoids from $R_{{\rm orbit}}$ to $R_u$. We begin with the purely algebraic considerations.

With the goal of showing that this map is well-defined in mind, suppose that $\alpha$ is a loop based at $y_0$ representing an element in $\pi_1(Y)$ (which we can and will assume is smooth) and $(\gamma_0 \pi_1(Y), \gamma_1 g_*(\pi_1(Y)), \gamma_2 g^2_*(\pi_1(Y)), \ldots )$ is an element in $\Omega$. Then 
\[
((\gamma_0 \pi_1(Y), \gamma_1 g_*(\pi_1(Y)), \ldots ), (\alpha \gamma_0 \pi_1(Y), \alpha \gamma_1 g_*(\pi_1(Y)), \ldots )) \in R_{{\rm orbit}}.
\]
Then, for each $i$, 
\[ d( (\hat{\alpha} \cdot \tilde{\gamma_i})(1)), \tilde{\gamma_i}(1)) \leq {\rm arclength}(\tilde{\alpha}), \]
where 
\begin{enumerate}
\item $\tilde{\gamma_i}$ is the unique lift of $\gamma_i$ with respect to $g^i$ starting at $y_0$;
\item $\hat{\alpha}$ is the unique lift of $\alpha$ with respect to $g^i$ starting at $\tilde{\gamma_i}(1)$;
\item $\tilde{\alpha}$ is the unique lift of $\alpha$ with respect to $g^i$ starting at $y_0$.
\end{enumerate}
It is worth noting that $(\hat{\alpha} \cdot \tilde{\gamma_i})$ is also the unique lift of the concatenation of $\alpha$ and $\gamma$ starting at $y_0$.

Since $\tilde{\alpha}$ is a lift of $\alpha$ with respect to $g^i$, $g^i \circ \tilde{\alpha} = \alpha$. Hence, 
\[ 
{\rm arclength}(\tilde{\alpha}) \leq \frac{\lambda^{-i}}{C} {\rm arclength}(\alpha).
\]
It follows that given $\epsilon>0$ there exists $N\in \N$ such that if $i\ge N$, then 
\[ d( (\hat{\alpha} \circ \tilde{\gamma_i})(1)), \tilde{\gamma_i}(1)) \leq {\rm arclength}(\tilde{\alpha}) \leq \frac{\lambda^{-i}}{C} {\rm arclength}(\alpha) < \epsilon. \]
Hence $(\Phi \times \Phi)(R_{{\rm orbit}}) \subseteq R_u$.

Next we show that $(\Phi^{-1} \times \Phi^{-1})(R_u) \subseteq R_{orbit}$. As such, let 
\[
((y_0, y_1, \ldots ), (z_0, z_1, \ldots) ) =(\yn, \zn) \in R_u.
\]
Fix $\epsilon>0$. We can and will assume that $\epsilon$ is small enough so that 
\begin{enumerate}
\item For each $y \in Y$, $B(y, \epsilon)$ is diffeomorphic to $\R^d$ and
\item $g$ evenly covers $B(y, \epsilon)$ for each $y\in Y$. 
\end{enumerate}
Since $(\yn, \zn) \in R_u$, there exists $i_0$ such that 
\[
d_Y( y_i , z_i ) < \epsilon \hbox{ for any }i\geq i_0.
\]
The two assumptions above imply that there exists a path $\alpha_{i_0} : [0, 1] \rightarrow Y$ starting at $y_{i_0}$ and ending at $z_{i_0}$ such that
\[
\alpha_{i_0}(t) \in B(y_{i_0}, \epsilon) \hbox{ for each }t\in [0,1]
\] 
For each $i>i_0$, define $\alpha_i$ inductively via
\[
\alpha_i= (g|_{B(y_{i-1}, \epsilon)})^{-1} \circ \alpha_{i-1}
\]
where the fact that $g$ evenly covers $B(y_i, \epsilon)$ has been used to define $(g|_{B(y_{i-1}, \epsilon)})^{-1} : B(y_{i-1}, \epsilon) \rightarrow U \subseteq B(y_i, \epsilon)$ (where $U$ is an open set in $Y$) and for simplicity we have assumed that $C \lambda<1$ (if this is not the case, then one uses a sufficiently large $K$ so that $C \lambda^K<1$ and applies our argument to $g^K$).

For each $i< i_0$, defined 
\[
\alpha_i = g^{i_0-i} \circ \alpha_{i_0}.
\]
By construction, for each $i$, $\alpha_i$ is a path from $y_i$ to $z_i$. In particular, $\alpha_0$ is a loop based at $y_0$ and hence defines an element in $[\alpha_0] \in \pi_1(Y)$. Furthermore, by construction, $\alpha_i$ is the unique lift of $\alpha_0$ starting at $y_i$ with respect to $g^i$. It follows from how the odometer action is defined that $[\alpha_0] \cdot \Phi^{-1}(\yn)=\Phi^{-1}(\xn)$.

This completes the proof that $\Phi \times \Phi$ is an isomorphism of groupoids, but we still need to consider the topologies involved. To do so, recall that the topology on the action groupoid associated to the odometer is given as follows. The groupoid is $\Omega \times \pi(Y)$ with the product topology (note that $\pi(Y)$ has the discrete topology). Using the fact that the action is free we have the identification with the relation $R_{{\rm orbit}}$ via $(a, \alpha) \mapsto (\alpha \cdot a, a)$. The map $\Phi^{-1}\times \Phi^{-1}$ at the level of topological groupoids becomes 
\[
(\yn, \zn)=((y_0, y_1, \ldots ), (z_0, z_1, \ldots) ) \mapsto ( (z_0, z_1, \ldots), [\alpha_0])
\]
where $\alpha_0$ is constructed as in the previous paragraphs and we have used the fact that $\Omega$ and $Z$ are homeomorphic to identify these spaces. We will show that this map is continuous and open. 

To show it is continuous, let $(\yn^k, \zn^k)_{k\in \N}$ be a sequence converging to $(\yn, \zn)$ in the domain. We will show that $(\zn^k, [\alpha_0^k])_{k\in \N}$ converges to $(\zn, [\alpha_0])$ in $\Omega \times \pi(Y)$. Let $U \times \{ [\alpha_0] \}$ be an open set containing $(\zn, [\alpha_0])$. Let $i_0$ be as in the construction of $\alpha_0$. By shrinking $U$, we can assume that for each $(a_0, a_1, \ldots ) \in U$, we have that $a_i=z_i$ for each $i\le i_0$. 

Again, by possibly shrinking $U$, one can form the following open neighbourhood of $(\yn, \zn)$: 
\[
V:= \{ ( h( a_0, a_1, \ldots), (a_0, a_1, \ldots)) \mid (a_0, a_1, \ldots) \in U \}
\]
where $h: U \rightarrow X^s(\yn, \epsilon_X)$ is defined as in Subsection \ref{subSecGroupoids}.

Since $(\yn^k, \zn^k)_{k\in \N}$ converges to $(\yn, \zn)$ there exists $K\in \N$ such that $(\yn^k, \zn^k)\in V$ for each $k\geq K$. Using this fact, we have that $z^k_i=z_i$ for each $i\le i_0$. Hence, $\alpha_0^k=\alpha_0$ for each $k\geq K$. Finally, $\Phi^{-1} \times \Phi^{-1}(\yn^k, \zn^k)=(\zn^k, [\alpha_0])$, so that $\Phi^{-1} \times \Phi^{-1}(\yn^k, \zn^k) \in U$ for each $k\geq K$ as required. 

Next, we must show that $\Phi^{-1} \times \Phi^{-1}$ is an open map. However, the proof is very similar to the previous argument so we omit it. 

In summary, we have shown the restriction of $R_u$ to $Z$ is isomorphic to $R_{orbit}$. Hence, $R_u$ and $R_{orbit}$ are Morita equivalent, see \cite[Section 3]{MR4030921}. 
\end{proof}

We can now prove Theorem \ref{KHunstable}. The equivalence in the previous lemma allows us to use known results about the $K$-theory of $C^*( \Omega \rtimes \pi_1(Y))$ to compute the $K$-theory of the unstable algebra and likewise for homology. 

For $K$-theory, see \cite{Scarparo} page 2544, we have that 
\[ C(\Omega) \rtimes \pi_1(Y) \cong \varinjlim C(\Omega_i) \rtimes \pi_1(Y) \]
where $\Omega_i=\pi_1(Y)/ g^i_*(\pi_1(Y))$ and the map in the inductive limit is obtained from the map $\Omega_{i+1} \rightarrow \Omega_i$ defined using $g^{i+1}_*(\pi_1(Y)) \subseteq g^i_*(\pi_1(Y))$. Then, \cite[Proposition 2.3]{Scarparo}, implies that, for each $i$, 
\[ C(\Omega_i) \rtimes \pi_1(Y) \cong M_{n^i}(\C) \otimes C^*_r(\pi_1(Y)) \] 
where the fact that $g$ is an $n$-fold cover and (for each $i$) $p^i_*(\pi_1(Y)) \cong \pi_1(Y)$ has been used. It follows that
\[
K_*(C(\Omega_i) \rtimes \pi_1(Y)) \cong K_*(M_{n^i}(\C) \otimes C^*_r(\pi_1(Y))) \cong K_*(C^*_r(\pi_1(Y))) \cong K_*(Y)
\]
where in the last step we have used the fact that $\pi_1(Y)$ satisfies the Baum--Connes conjecture, $\pi_1(Y)$ is torsion-free, and $Y$ is a model for $B(\pi_1(Y))$.

Continuing, we determine the maps in the inductive system. The Baum--Connes conjecture with coefficients implies that for each $i$, 
\[
K_*(C(\Omega_i) \rtimes \pi_1(Y)) \cong KK^{\pi_1(Y)}_*(C_0(\R^d), C(\Omega_i))
\]
and the connecting maps in the inductive limit are given by 
\[ (g_i)_* : KK^{\pi_1(Y)}_*(C_0(\R^d), C(\Omega_i)) \rightarrow KK^{\pi_1(Y)}_*(C_0(\R^d), C(\Omega_{i+1})) \]
where at the space level $g_i: \Omega_{i+1} \rightarrow \Omega_i$ is defined using coset inclusion (i.e., $g^{i+1}(\pi_1(Y)) \subseteq g^i_*(\pi_1(Y))$). It is worth noting that the map $g_i$ is a covering map. The fact that the $C^*$-algebra $C(\Omega_i)$ is equivariantly Poincar\'e self-dual along with properties of transfer maps then implies that we have the following commutative diagram:
\[ \begin{CD}
		KK_*^{\pi_1(Y)}(C_0(\R^d), C(\Omega_i)) @>PD>> KK_*^{\pi_1(Y)}(C_0(\R^d \times \Omega_i), \C) \\
		@V(g_i)_* VV @V(g_i)! VV \\
		KK_*^{\pi_1(Y)}(C_0(\R^d), C(\Omega_{i+1}) @>PD>>KK_*^{\pi_1(Y)}(C_0(\R^d \times \Omega_{i+1}), \C)\\
\end{CD} \]
Furthermore at the space level, there is a commutative diagram
\[ \begin{CD}
		(\R^d \times \Omega_{i+1}) / \pi_1(Y) @>\cong >> Y \\
		@Vg_i VV @Vg VV \\
		(\R^d \times \Omega_i ) / \pi_1(Y) @>\cong >> Y \\
\end{CD} \]
where the horizontal maps are homeomorphisms. This leads to the commutative diagram
\[ \begin{CD}
		KK_*^{\pi_1(Y)}(C_0(\R^d \times \Omega_i), \C) @> \cong >> K_*((\R^d \times \Omega_i ) / \pi_1(Y)) @> \cong >> K_*(Y)  \\
		@V(g_i)! VV @. @Vg! VV \\
		KK_*^{\pi_1(Y)}(C_0(\R^d \times \Omega_{i+1}), \C) @>\cong >> K_*((\R^d \times \Omega_{i+1} ) / \pi_1(Y)) @> \cong >> K_*(Y)  \\
\end{CD} \]
which completes the argument for $K$-theory. For the statement about homology, one uses \cite[Proposition 2.4]{Scarparo}.


\begin{remark}

Using the fact that the homoclinic groupoid is Morita equivalent to $G^s(P) \times G^u(P)$, one can compute the $K$-theory of the homoclinic algebra and the homology of the homoclinic groupoid using Theorems \ref{KHstable} and \ref{KHunstable} along with the relevant K\"unneth formula.

One can also compute the $K$-theory of the Ruelle algebras using the above results and Theorem 5.10 in \cite{MR3868019}.

\end{remark}

\begin{remark}
It follows from the Morita equivalence in the previous theorem and the main result in \cite{DeeHKConj} that there is a Smale space whose unstable groupoid is a counterexample to the HK-conjecture.
\end{remark}

\subsection{General results}
Some general results that follow from the theorems in the previous section are discussed.
\begin{theorem} \label{mainResultZeroDegree}
Suppose $Y$ is a flat manifold (which is connected), $g : Y \rightarrow Y$ is an expanding endomorphism (which is an $n$-fold cover), $(X, \varphi)$ is the associated Smale space, $G^s(P)$ (respectively $G^u(P)$) is the stable (respectively unstable) groupoid of $(X, \varphi)$. Then 
\[ H_0(G^s(P)) \cong H_0(G^u(P)) \cong \Z\left[ \frac{1}{n} \right],\]
and the map induced from $\varphi^{-1}$ on $H_0(G^s(P))\otimes \Q$ and from $\varphi$ on $H_0(G^u(P)\otimes \Q$ are each given by multiplication by $n$. Moreover, the $K$-theory groups in degree zero of the stable and unstable algebras each contain $\Z\left[ \frac{1}{n} \right]$ as a factor.
\end{theorem}
\begin{proof}
Since $Y$ is connected, $H_0(Y) \cong H^0(Y) \cong \Z$. In both cases, the induced map from $g$ is given by the identity and hence the transfer map associated to $g$ (again in both cases) is given by multiplication by $n$. Theorems \ref{KHstable} and \ref{KHunstable} then imply that 
\[ H_0(G^s(P)) \cong H_0(G^u(P)) \cong \Z\left[ \frac{1}{n} \right].\]
To compute the induced map on homology, by \cite{MR0298701}, there is an s-bijective map from the full $n$ shift to $(X, \varphi)$. In fact, 
$H_0(G^u(P))\otimes \Q \cong D^s( \Sigma_n) \otimes \Q$ via the map induced from this s-bijective map. Finally, by functorial properties of Putnam's homology theory, there is a commutative diagram
\[ \begin{CD}
		D^u(\Sigma_n)\otimes \Q @>>> H_0(G^s(P))\otimes \Q  \\
		@V\sigma^u VV @V\varphi^u VV \\
		D^u(\Sigma_n)\otimes \Q @>>> H_0(G^s(P))\otimes \Q, \\
\end{CD} \]
and likewise with unstable replaced by stable. The result then follows since the induced map for the full $n$ shift is given by multiplication by $n$.

The proof of the statement about $K$-theory is similar, so we omit the details.
\end{proof}

\begin{theorem} \label{mainResultTopDegree}
Suppose $Y$ is a flat manifold (which is connected has dimension $d$), $g : Y \rightarrow Y$ is an expanding endomorphism, $(X, \varphi)$ is the associated Smale space, $G^s(P)$ (respectively $G^u(P)$) is the stable (respectively unstable) groupoid of $(X, \varphi)$. Then 
\[ H_d(G^s(P)) \cong H_d(G^u(P)) \cong \left\{ \begin{array}{ll}  \Z & \hbox{ if }Y\hbox{ is orientable} \\ \{ 0 \} & \hbox{ if }Y\hbox{ is not orientable} \end{array} \right., \]
and the map induced from $\varphi^{-1}$ on $H_0(G^s(P))$ and from $\varphi$ on $H_0(G^u(P))$ are each given by the identity map or its negation.
\end{theorem}
\begin{proof}
The proof is similar to the proof of the previous theorem, so the details are therefore omitted.
\end{proof}

\begin{theorem} \label{mainResultQ}
Suppose $Y$ is a flat manifold and $g: Y\rightarrow Y$ is an expanding endomorphism. Then 
\begin{align*}
K_*(C^*(G^s(P))) \otimes \Q & \cong K^*(Y)\otimes \Q \\
H_*(G^s(P)) \otimes \Q & \cong H^*(Y) \otimes \Q. \\
K_*(C^*(G^u(P))) \otimes \Q & \cong K_*(Y)\otimes \Q \\
H_*(G^u(P)) \otimes \Q & \cong H_*(Y) \otimes \Q.
\end{align*}
\end{theorem}
\begin{proof}
The proof follows from the fact (see Proposition \ref{propTransferProp}) that the transfer map (in all relevant cases) is a rational isomorphism for finite order self-covers and the inductive limits given in Theorems \ref{KHstable} and \ref{KHunstable}.
\end{proof}

\section{Computations} \label{SecComputations}
\subsection{A question of Putnam}
In chapter 8 of \cite{PutHom}, Putnam asks a number of questions about his homology theory. For example, Question 8.1.1 is addressed in \cite{MR3646890} and Question 8.4.1 is addressed in \cite{PV2}. We will address Question 8.3.2, which asks the following:
\begin{question} \label{IanQuestion}
If the stable or unstable sets in a Smale space $(X, \varphi)$ are
contractible in their standard topology, then does there exist an integer, $k$, such that
\[ H^*(X) \cong H_{*-k}(G^u(P))? \]
\end{question}
For Smale spaces associated with certain substitution tiling systems the answer to this question is yes, see \cite[Remark 4.5]{PV2}. However, we will give an example that shows that the answer to this question is no. There is a reformulation of Putnam's question that does have a positive answer, see \cite{PV3} for details. Oversimplifying, the reformulation in \cite{PV3} takes orientation into account. Before presenting the counterexample, some general facts are needed.
\begin{theorem} \label{CechX}
Suppose that $(X, \varphi)$ is the Smale space obtained from an expanding endomorphism $g : Y \rightarrow Y$ where $Y$ is a flat manifold as in Definition \ref{solenoidGroupoidDefn}. Then the unstable sets of $(X, \varphi)$ with their standard topology are homeomorphic to $\R^d$, where $d$ is the dimension of $Y$ and 
\[
H^*(X) \cong \lim (H^*(Y), g^*)
\] 
where $g^*$ is the map on cohomology associated to $g$.
\end{theorem}
\begin{proof}
From \cite{Wil}, $X^u(x, \epsilon) \cong \mathbb{D}_d$ where $\mathbb{D}_d$ is the open disk of dimension $d={\rm dim}(Y)$. It then follows from the decomposition of the global unstable set discussed in Section \ref{SmaleSec} as a nested union that, for each $x\in X$, $X^u(x) \cong \R^d$. The second half of the statement follows since \v{C}ech cohomology is continuous with respect to inverse limits.
\end{proof}

\begin{example} 
Let $Y$ be the Klein bottle and $g: Y \rightarrow Y$ be the expanding endomorphism give by the nine fold cover described in Diagram \ref{dia9Klein}. Also, let $(X, \varphi)$ denote the Smale space associated to $(Y, g)$. The Smale space $(X, \varphi)$ has contractible unstable sets by Theorem \ref{CechX}. The cohomology of the Klein bottle is
\[
H^*(Y) \cong \left\{ \begin{array}{ll} \Z & *=0 \\ \Z & *=1 \\ \Z/2\Z & *=2 \\ \{ 0 \} & \hbox{else.} \end{array} \right.
\]
The map on cohomology associated to $g$ (denoted by $g^*$) is given by 
\[
g^*: H^*(Y) \rightarrow H^*(Y) = \left\{ \begin{array}{l} \hbox{ the identity on }H^0 \\ \hbox{ multiplication by }3 \hbox{ on }H^1 \\ \hbox{ the identity on } H^2. \end{array} \right.
\]
Using this formula for $g^*$ and Proposition \ref{propTransferProp}, the transfer map associated to $g$ is given by 
\[
t_{cohomology}: H^*(Y) \rightarrow H^*(Y) = \left\{ \begin{array}{l} \hbox{ multiplication by }9 \hbox{ on }H^0 \\ \hbox{ multiplication by }3 \hbox{ on }H^1 \\ \hbox{ the identity on } H^2. \end{array} \right.
\]
Using Theorem \ref{CechX}, 
\[ 
H^*(X) \cong \left\{ \begin{array}{ll} \Z & *=0 \\ \Z\left[ \frac{1}{3} \right] & *=1 \\ \Z/2\Z & *=2 \\ \{ 0 \} & \hbox{else.} \end{array} \right.
\]
Using Theorem \ref{KHstable}, 
\[ 
H^u_*(X, \varphi) \cong H_*(G^s(P)) \cong \left\{ \begin{array}{ll} \Z\left[ \frac{1}{9} \right] & *=0 \\ \Z\left[ \frac{1}{3} \right] & *=1 \\ \Z/2\Z & *=2 \\ \{ 0 \} & \hbox{else.} \end{array} \right.
\]
It follows from these computations that the answer to Question \ref{IanQuestion} is no.

One can likewise show that 
\[ 
H^s_*(X, \varphi) \cong H_*(G^u(P)) \cong \left\{ \begin{array}{ll} \Z\left[ \frac{1}{9} \right] & *=0 \\ \Z\left[ \frac{1}{3} \right] \oplus \Z/2\Z & *=1 \\ \{ 0 \} & \hbox{else,} \end{array} \right.
\]
starting with the fact that the homology of the Klein bottle is
\[
H_*(Y) \cong \left\{ \begin{array}{ll} \Z & *=0 \\ \Z \oplus \Z/2\Z  & *=1 \\  \{ 0 \} & \hbox{else.} \end{array} \right.
\]
\end{example}

On the positive side, we do have the following.
\begin{theorem} \label{orientablePutnamQuestion}
Suppose $g : Y \rightarrow Y$ is an expanding endomorphism where $Y$ is an orientable flat manifold and $(X, \varphi)$ is the associated Smale space. Then 
\[ H^*(X) \cong H^s_{d-*}(X, \varphi) \cong H_{d-*}(G^u(P)) \]
where $d$ is the dimension of $Y$.
\end{theorem}
\begin{proof}
Since $Y$ is orientable, Poincar\'e duality holds. Moreover, the following diagram commutes:
\[ \begin{CD}
		H^*(Y) @>g^* >> H^*(Y) \\
		@VPD VV @VPD VV \\
		H_{d-*}(Y) @>t_{{\rm homology}} >> H_{d-*}(Y), \\
\end{CD} \]
where the vertical maps are the Poincar\'e duality isomorphisms. The result then follows from Theorem \ref{KHunstable} and Theorem \ref{CechX}.
\end{proof}
There is a version of the previous theorem in the context of $K$-theory:
\begin{theorem} \label{spincPutnamQuestion}
Suppose $g : Y \rightarrow Y$ is an expanding endomorphism where $Y$ is a spin$^c$ flat manifold and $(X, \varphi)$ is the associated Smale space. Then 
\[ K^*(X) \cong K_{d-*}(C^*(G^u(P))) \]
where $d$ is the dimension of $Y$ and indices are modulo two.
\end{theorem}
\begin{proof}
Since $Y$ is spin$^c$, Poincar\'e duality holds with respect to $K$-theory and $K$-homology. Moreover, the following diagram commutes:
\[ \begin{CD}
		K^*(Y) @>g^* >> K^*(Y) \\
		@VVV @VVV \\
		K_{d-*}(Y) @>t_{{\rm K-homology}} >> K_{d-*}(Y) \\
\end{CD} \]
where the vertical maps are the Poincar\'e duality isomorphisms. The result then follows from Theorem \ref{KHunstable} and Theorem \ref{CechX}.
\end{proof}
Putnam's question also applies to $G^s(P)$. However, in this case relating the inductive limit in Theorem \ref{KHstable} to the algebraic topology of $X$ is less clear because homology does not respect inverse limits. Nevertheless the next two theorems are useful because the induced map is usually easier to determine than the transfer map. They also indicate that the homology of $G^s(P)$ and the $K$-theory of its $C^*$-algebra are easier to compute than the homology of $X$.
\begin{theorem}
Suppose $g : Y \rightarrow Y$ is an expanding endomorphism where $Y$ is an orientable flat manifold and $(X, \varphi)$ is the associated Smale space. Then 
\[ H^u_{*}(X, \varphi) \cong H_{*}(G^s(P)) \cong \lim (H_{*-d}(Y), g_*) \]
where $d$ is the dimension of $Y$.
\end{theorem}
\begin{theorem}
Suppose $g : Y \rightarrow Y$ is an expanding endomorphism where $Y$ is a spin$^c$ flat manifold and $(X, \varphi)$ is the associated Smale space. Then 
\[ K_{*}(C^*(G^s(P))) \cong \lim (K_{*-d}(Y), g_*) \]
where $d$ is the dimension of $Y$ and indices are modulo two.
\end{theorem}
The proof of the previous two theorems are very similar to the proofs of Theorems \ref{orientablePutnamQuestion} and \ref{spincPutnamQuestion} and therefore are omitted.

\subsection{Torsion in Putnam's homology}
At present there is little known about the range of Putnam's homology theory. The following is known: Putnam's homology is non-trivial for only finitely many indices, has finite rank, often is not finitely generated, and can contain torsion. In this section we explore the torsion subgroup.
\begin{theorem} \label{SameTorsion}
Suppose $Y$ is a flat manifold. Then there exists an expanding endomorphism $g: Y \rightarrow Y$ such that
\[
T(H_*(Y)) \cong T(H_*(G^u(P))) \hbox{ and } T(H^*(Y)) \cong T(H_* (G^s(P)))
\]
where
\begin{enumerate}
\item $(X, \varphi)$ is the Smale space associated to $(Y,g)$ and
\item $T(G)$ denotes the torsion subgroup of an abelian group $G$.
\end{enumerate}
\end{theorem}
\begin{proof}
The proof is similar to the proof of \cite[Theorem 4.4]{DeeHKConj}. By the main result of \cite{EpSh} (see the theorem on page 140 of \cite{EpSh} or \cite[Theorem 4.4]{DeeHKConj} for details), there exists an expanding endomorphism $g: Y\rightarrow Y$ such that $g$ is an $n$-fold cover with $n=(|F|+1)^d$ where 
\begin{enumerate}
\item $d$ is the dimension of $Y$,
\item $F$ is the holonomy group of $\pi_1(Y)$, see Subsection \ref{ExpEndSec}, and
\item $|F|$ is the order of the finite group $F$.
\end{enumerate}
By Theorem \ref{propTransferProp}, we have that
\[
g_* \circ t_{homology} = \hbox{ multiplication by }(|F|+1)^d,
\]
and likewise
\[
t_{cohomology} \circ g^*=\hbox{ multiplication by }(|F|+1)^d.
\]
By Proposition \ref{propTransferProp}, multiplication by $(|F|+1)^d$ is the identity on the torsion subgroups of (co)homology of $Y$. Since the torsion subgroups of (co)homology of $Y$ are finite, it follows that $t_{homology}$ and $t_{cohomology}$ are isomorphisms when restricted to the torsion part (co)homology. The result now follows from the inductive limits in Theorems \ref{KHstable} and \ref{KHunstable}.
\end{proof}
In other words, the previous theorem states that whatever torsion occurs in the homology/cohomology of a flat manifold also occurs in Putnam's homology. This result can often be used to give positive answers to questions about Putnam's homology theory via results about flat manifolds. For example, we have the following:
\begin{theorem}
Suppose $p$ is a prime. Then there exists a Smale space $(X, \varphi)$ such that $T(H^u_1(X, \varphi))\cong \Z/p\Z$.
\end{theorem}
\begin{proof}
Theorem \ref{SameTorsion} reduces the proof to showing that there exists a flat manifold $Y$ with $T(H_1(Y)) \cong \Z/p\Z$. The existence of the required flat manifold follows from \cite[Theorem 3.10(i)]{MR170305}. 
\end{proof}

\subsection{Further computations}
For each $d\in \N$ there are finitely many flat manifolds with dimension $d$. For $d=3$, there are ten, and a complete list with their homologies is as follows:
\newpage
\begin{center}
{\bf Orientable flat $3$-manifolds}
\vspace{0.3cm}

\begin{tabular}{c|c|c|c|c}
Manifold & $H_1$ & $H_2$ & $H_3$ & Holonomy\\
\hline
$O^3_1$ & $\Z^3$ & $\Z^3$ & $\Z$ &  $\{0\}$ \\
\hline
$O^3_2$ & $\Z\times(\Z_2)^2$ & $\Z$ & $\Z$  &  $\Z_2$\\
\hline
$O^3_3$ & $\Z\times\Z_3$ & $\Z$ &  $\Z$ & $\Z_3$\\
\hline
$O^3_4$ & $\Z\times\Z_2$ & $\Z$ & $\Z$ & $\Z_4$\\
\hline
$O^3_5$ & $\Z$ & $\Z$ & $\Z$ & $\Z_6$\\
\hline
$O^3_6$ & $(\Z_4)^2$ & $\{ 0 \}$ & $\Z$ & $(\Z_2)^2$\\
\end{tabular}
\vspace{0.3cm}

{\bf Nonorientable flat $3$-manifolds}
\vspace{0.3cm}

\begin{tabular}{c|c|c|c|c}
Manifold & $H_1$ & $H_2$ & $H_3$ & Holonomy\\
\hline
$N^3_1$ & $\Z^2\times\Z_2$ & $\Z^2\times \Z_2$ & $\{0\}$ & $\Z_2$ \\
\hline
$N^3_2$ & $\Z^2$ & $\Z^2 \times \Z_2$ & $\{0 \}$ & $\Z_2$\\
\hline
$N^3_3$ & $\Z\times(\Z_2)^2$ & $\Z \times \Z_2$ & $\{0 \}$ & $(\Z_2)^2$\\
\hline
$N^3_4$ & $\Z\times\Z_4$ & $\Z \times \Z_2$ & $\{ 0 \}$ & $(\Z_2)^2$\\
\end{tabular}
\vspace{0.3cm}
\end{center}
The reader can find this list in \cite{MR0217740} (where we note that Wolf only lists the $H_1$-group but the other homology groups can be computed from it and the other information in the list from \cite{MR0217740}). Using the list, one can compute the Putnam homology and $K$-theory for many explicit examples. First, one can compute the cohomology groups; either using Poincar\'e duality when the flat manifold is orientable or the universal coefficient theorem when the flat manifold is non-orientable. Then, using the fact that we are working in dimension three, the cohomology/homology determine the $K$-theory/$K$-homology. Namely, 
\[
K^*(Y) \cong \bigoplus_{*+2i} H^*(Y) \hbox{ and } K_*(Y) \cong \bigoplus_{*+2i} H_*(Y),
\]
see for example \cite{MR1951251} for details. Next, the relevant transfer maps can be computed, usually using Proposition \ref{propTransferProp} Part (2). Finally, the relevant inductive limits can be determined.

There are lists of flat manifolds in other dimensions and lists when one assumes a specific holonomy group, see the results in \cite{MR170305, MR189064} in the case when the holonomy is cyclic of prime order. The process discussed in the case of three dimensional flat manifolds can be generalized and applied to such lists. As a sample we discuss the case of Hantzsche--Wendt manifolds. The reader can find work on Hantzche-Wendt manifolds in for example \cite{MR2581917}.

\begin{definition}
A Hantzsche--Wendt manifold is an orientable flat manifold of dimension $d$ with holonomy group $(\Z/2\Z)^{d-1}$.
\end{definition}
\begin{example}
In dimension three, there is one Hantzsche--Wendt manifold. It is $O^3_6$ in the list above.
\end{example}
Some important facts we will use are the following.
\begin{enumerate}
\item If $Y$ is a  Hantzsche--Wendt manifold then its dimension is odd and for any $d\ge 3$ odd there exists a Hantzsche--Wendt manifold. 
\item The rational homology of a $d$ dimensional Hantzsche--Wendt manifold is the same as the $d$-sphere. In particular, if $Y$ is a $d$ dimensional Hantzsche--Wendt manifold, then it satisfies $H_0(Y) \cong H_d(Y) \cong \Z$ and all other homology groups are pure torsion. 
\end{enumerate}

\begin{theorem}
Suppose $d\ge 3$ is an odd integer, $Y$ is a $d$ dimensional Hantzsche--Wendt manifold, $g: Y\rightarrow Y$ is an  expanding endomorphism as in Theorem \ref{SameTorsion}, and $(X, \varphi)$ is the Smale space associated to $(Y,g)$. Then 
\[
H_*(G^u(P)) \cong \left\{ \begin{array}{ll} \Z \left[ \frac{1}{(2^{d-1}+1)^d} \right] & *=0 \\ H_n(Y) & \hbox{ else} \end{array} \right.
\]
and
\[
H_*(G^s(P)) \cong \left\{ \begin{array}{ll} \Z \left[ \frac{1}{(2^{d-1}+1)^d} \right] & *=0 \\ H^n(Y) & \hbox{ else.} \end{array} \right.
\]
\end{theorem}
\begin{proof}
This result follows directly from the comment on the homology of a Hantzsche--Wendt manifold stated just before the theorem, Theorem \ref{SameTorsion}, Theorem \ref{mainResultZeroDegree}, and Theorem \ref{mainResultTopDegree}.
\end{proof}

\begin{theorem}
Suppose $d\ge 3$ is an odd integer, $Y$ is a $d$ dimensional Hantzsche--Wendt manifold, $g: Y\rightarrow Y$ is an expanding endomorphism that is an $n$-fold cover, and $(X, \varphi)$ is the Smale space associated to $(Y,g)$. Then, for each $k\in \N$,
\[
|{\rm Per}_k(X, \varphi)| = n^k-1 \hbox{ or } n^k+1.
\]
\end{theorem}
\begin{proof}
By the previous theorem, 
\[
H_*(G^u(P))\otimes \Q \cong \left\{ \begin{array}{ll} \Q & *=0 \hbox{ or }d \\ \{ 0 \} & \hbox{ else.} \end{array} \right.
\]
and by Theorems \ref{mainResultZeroDegree} and \ref{mainResultTopDegree}, the induced maps are multiplication by $n$ in degree zero and the identity or its negation in degree $d$. The result then follows from Putnam's Lefschetz theorem \cite[Theorem 6.1.1]{PutHom}.
\end{proof}

\end{document}